\documentclass[11pt]{amsart}
\usepackage{amssymb,amsfonts,amsmath}
\usepackage[all]{xy}
\usepackage[shortlabels]{enumitem}
\usepackage{hyperref, color}		
\newtheorem{theorem}{Theorem}[section]
\newtheorem{proposition}{Proposition}[section]

\newtheorem{corollary}{Corollary}[section]
\theoremstyle{definition}
\newtheorem{definition}{Definition}[section]
\newtheorem{lemma}{Lemma}[section]
\newtheorem{example}{Example}[section]
\theoremstyle{remark}
\newtheorem{remark}{Remark}
\theoremstyle{claim}

\theoremstyle{theoremA}
\newtheorem*{theoremA}{Theorem A}
\theoremstyle{theoremB}
\newtheorem*{theoremB}{Theorem B}  
\newcommand{\real}{\mathbb{R}}

\newcommand{\la}{\lambda}
\newcommand{\La}{\Lambda}
\newcommand{\al}{\alpha}
\newcommand{\om}{\omega}
\newcommand{\Om}{\Omega}
\newcommand{\na}{\nabla}
\newcommand{\ep}{\epsilon}

\newcommand{\Ga}{\Gamma}
\newcommand{\be}{\beta}
\newcommand{\ka}{\kappa}

\newcommand{\de}{\delta}

\newcommand{\ti}{\tilde}
\newcommand{\vol}{\text{\rm vol}}

\newcommand{\lan}{\left\langle}
\newcommand{\ran}{\right\rangle}
\newcommand{\p}{\partial}

\newcommand{\dv}{\mathrm{div\,}}
\newcommand{\m}{\mathcal}
\newcommand{\tr}{\mathrm{tr}\,}

\newcommand{\II}{I\!I}
\newcommand{\hs}{\mathrm{Hess}}
\newcommand{\rk}{\mathrm{rk}}
\newcommand{\rad}{\mathrm{rad}}

\begin{document}


\title[Integral estimates for the trace of symmetric operators]{Integral estimates for the trace of symmetric operators.}
\author{M. Batista}
\address{Instituto de Matem\'atica, Universidade Fe\-deral de Alagoas, Macei\'o, 
AL, CEP 57072-970, Brazil.}\email{mhbs28@gmail.com}
\author{H. Mirandola}
\address{Instituto de Matem\'atica, Universidade Fe\-deral do Rio de Janeiro, 
Rio de Janeiro, RJ, CEP 21945-970, Brazil.} \email{mirandola@ufrj.br}
\date{\today}
\maketitle
\thispagestyle{empty}
%


\begin{abstract} Let $\Phi:TM\to TM$ be a positive-semidefinite symmetric operator of class $C^1$ defined on a complete non-compact manifold $M$ isometrically immersed in a Hadamard space $\bar{M}$. In this paper, we given conditions on the operator $\Phi$ and on the second fundamental form to guarantee that either $\Phi\equiv 0$ or the integral $\int_M \mathrm{tr}\,\Phi dM$ is infinite. We will given some applications. The first one says that if $M$ admits an integrable distribution whose integrals are  minimal submanifolds in $\bar{M}$ then the volume of $M$ must be infinite. Another application states that if the sectional curvature of $\bar{M}$ satisfies $\bar{K}\leq -c^2$, for some $c\geq 0$, and $\lambda:M^m\to [0,\infty)$ is a nonnegative $C^1$ function such that gradient vector of $\lambda$ and the mean curvature vector $H$ of the immersion satisfy $|H+p\nabla \lambda|\leq (m-1)c \lambda$, for some $p\geq 1$, then either $\lambda\equiv 0$ or the integral $\int_M \lambda^s dM$ is infinite, for all $1\leq s\leq p$. 
\end{abstract}


\section{Introduction} 

Let $f:M^m\to \bar M$ be an isometric immersion of an $m$-dimensional Riemannian manifold $M$ in a Riemannian manifold $\bar M$ and $\II$ the second fundamental form of $f$. Let $\Phi:TM\to TM$ be a symmetric operator on $M$ of class $C^1$.  Consider the following definitions:
\begin{definition} The {\it $\Phi$-mean curvature vector field of the immersion} $f$ is the normal vector field $H_{\Phi}:M\to T^\perp M$ to the immersion $f$ defined by the trace:
\begin{equation*}
H_{\Phi}=\tr\left\{(X,Y)\in TM\times TM \mapsto \II(\Phi X,Y) \right\}
\end{equation*}
Note that $H_{\Phi}$ coincides with the mean curvature vector if $\Phi$ is the identity operator $I:TM\to TM$. 
\end{definition}
\begin{definition} The {\it divergence of $\Phi$} is the tangent vector field on $M$ defined by
\begin{equation*}
\lan \dv\Phi,X\ran = \tr\{Y\in TM \mapsto (\na_Y\Phi)X=\na_Y (\Phi X)-\Phi(\na_YX)\}, 
\end{equation*}
for all tangent vector field $X:M\to TM$. Note that if $\Phi=\la I$, where $\la:M\to \real$ is a $C^1$ function, then $\dv \Phi$ coincides with the gradient vector of $\la$.

\end{definition}
It is a well known fact that a complete noncompact minimal submanifold in a Hadamard manifold must have infinite volume (see for instance \cite{fr}). Our first theorem says the following:
\begin{theorem}\label{thm_asz} Let $f:M\to \bar M$ be an isometric immersion of a complete noncompact manifold $M$ in a complete simply-connected manifold $\bar M$ with nonpositive radial curvature with respect to some basis point  $q_0\in f(M)$.  Let $\Phi:TM\to TM$ be a positive-semidefinite symmetric operator of class $C^1$ such that $\tr\Phi(q_0)>0$. Assume that  
\begin{equation*}
|H_\Phi + \dv \Phi|\leq \frac{1}{r+\ep}\, \tr\Phi,
\end{equation*} 
for some $\ep>0$, where $r=d_{\bar M}(\cdot\,,q_0)$ is the distance in $\bar M$ from $q_0$. Then the rate of growth of the integral $\int_M\tr\Phi$ is at least logarithmic with respect the geodesic balls centered at $q_0$, that is, $$\liminf_{\mu\to \infty}\frac{1}{\log(\mu)}\int_{B_\mu(q_0)}\tr\Phi >0,$$
where $B_\mu$ denote the geodesic balls of $M$ centered at $q_0$. In particular, the integral $\int_M \tr\Phi$ must be infinite.
\end{theorem}

Before we enunciate the next results, we will to consider a consequence of Theorem \ref{thm_asz}. Let $M$ be a manifold with nonpositive radial curvature with respect to some base point $q_0$. It is easy to show that the radial curvature of the product manifold $\bar M=M\times \real$ with respect to the base point $(q_0,0)$ is also nonnegative. Furthermore, the inclusion map $j:M\to \bar M$ given by $j(x)=(x,0)$ is a totally geodesic embedding. Thus the result below follows directly from Theorem \ref{thm_asz}.
\begin{corollary} Let $M$ be a complete simply-connected manifold with nonpositive radial curvature with respect to some base point $q_0$. Let $\Phi:TM\to TM$ be a positive-semidefinite symmetric operator such that $\tr\Phi(q_0)>0$. Assume that  $$|\dv\Phi|\leq \frac{1}{r+\ep}\tr\Phi,$$ for some $\ep>0$, where $r=d_{\bar M}(\cdot\,,q_0)$. Then the rate of growth of the integral $\int_M \tr\Phi$ is at least logarithmic with respect to the geodesic balls centered at $q_0$. In particular, the integral $\int_M \tr\Phi$ is infinite.
\end{corollary}
The next theorem says the following:
\begin{theorem}\label{cor-asz-phi-min} Let $f:M\to \bar M$ be an isometric immersion of a complete noncompact manifold $M$ in a Hadamard manifold $\bar M$. Let $\Phi:TM\to TM$ be a positive-semidefinite symmetric operator of class $C^1$. Assume that $$\dv\Phi=H_{\Phi}=0.$$ Then, for all $q\in M$, we have that either $\Phi(q)=0$ or the rate of growth of the integral $\int_M\tr\Phi$ is at least linear with respect to the geodesic balls of $M$ centered at $q$, that is, $$\liminf_{\mu\to \infty}\frac{1}{\mu}\int_{B_\mu(q)}\tr\Phi >0,$$ where $B_\mu(q)$ denotes the geodesic ball of $M$ of radius $\mu$ and centered at $q$. In particular, either $\Phi$ vanishes identically or the integral $\int_M\tr\Phi$ is infinite.
\end{theorem}
The following result is a non-direct application of Theorem \ref{cor-asz-phi-min}. It will be proved in section \ref{sec-the-al} below.

\begin{corollary}\label{corol-dist} Let $f:M\to \bar M$ be an isometric immersion of a complete noncompact manifold $M$ in a complete simply-connected manifold $\bar M$ with nonpositive radial curvature with respect to some base point in $f(M)$. Let $\m D$ be a $k$-dimensional integrable distribution on $M$, with $k\geq 1$, such that each integral of $\m D$ is a minimal submanifold in $\bar M$. Then the rate of volume growth of $M$ is at least linear.
\end{corollary}
To state our next applications of Theorems \ref{thm_asz} and \ref{cor-asz-phi-min} we need to recall some notations. Let $B:TM\to TM$ be a symmetric operator of class $C^1$. We recall that $B$ satisfies the Codazzi equation if the following holds:
\begin{equation*}
(\na_X B)Y=(\na_Y B)X,
\end{equation*}
for all tangent vector fields $X$ and $Y$ on $M$. The Newton operators $P_j=P_j(B)$, $j=1,\ldots,m$, associated to $B$, are the symmetric operators on $M$ defined inductively by:
\begin{equation*}
\begin{array}{l}
P_0=I;\\
P_{j}=S_jI-BP_{j-1}, \mbox{ with } j\geq 1,
\end{array}
\end{equation*}
where $S_j=S_j(B)=\sum_{i_1<\ldots<i_r} \la_{i_1}\ldots \la_{i_j}$ is the jth-elementary symmetric polynomial evalued on the eigenvalues $\la_1,\ldots, \la_m$ of $B$. 

A result due to Alencar, Santos and Zhou \cite{asz} says the following:
\begin{theoremA}[Corollary 2.2 of \cite{asz}] Let $f:M^m\to Q_c^{m+1}$ be a noncompact properly immersed hypersurface in a space form $Q_c^{m+1}$, with $c\leq 0$. Let $P_j=P_j(A)$, $j=1,2,\ldots$, be the Newton operators associated to the shape operator $A$ of the immersion $f$. Assume that $$S_j\geq 0 \ \mbox{ and } \ S_{j+1}=0,$$ for some $1\leq j\leq m-1$. Then either $S_j\equiv 0$ or the integral $\int_M S_j$ is infinite.
\end{theoremA}
Actually, under the hypotheses of Theorem A, Alencar, Santos and Zhou \cite{asz} proved that, for each $q\in M$, it holds that either $S_j(q)=0$ or the rate of growth of the integral $\int_M S_j$ is at least linear with respect to the geodesic balls of $M$ centered at $q$. 

The result below is a non-direct consequence of Theorem \ref{thm_asz}. It will be proved in section \ref{sec-the-al}.

\begin{theorem}\label{cor-asz} Let $f:M^m\to \bar M$ be a complete non-compact isometric immersion in a complete simply-connected manifold $\bar M$ with nonpositive radial curvature with respect to some base point $q_0\in f(M)$. Let $P_j=P_j(B)$, $j=1,2,\ldots$, be the Newton operators associated to a symmetric operator $B:TM\to TM$ that satisfies the Codazzi equation. Assume that $$S_{j+1}(B)= 0, \ \ S_j(B(q_0))\neq 0 \ \mbox{ and } \  S_j(B) \mbox{ does not change of sign,}$$ for some $1\leq j\leq m-1$. Assume further that the $P_j$-mean curvature vector satisfies $$|H_{P_j}|\leq \frac{1}{r+\ep},$$ where $r=d_{\bar M}(\cdot\,,q_0)$. Then the rate of growth of the integral $\int_M |S_j(B)|$ is at least logarithmic with respect to the geodesic balls of $M$ centered at $q_0$.
\end{theorem}

Note that if $f:M^m\to Q_c^{m+1}$ is a hypersurface and $A:TM\to TM$ is the shape operator then the Newton operator $P_j=P_j(A)$ satisfies $$H_{P_j}=\tr(AP_j)=(j+1)S_{j+1}.$$ 
Thus the result below improves Theorem A and it follows as a consequence of Theorem \ref{cor-asz-phi-min}.

\begin{theorem}\label{cor-asz2} Let $f:M^m\to Q^{m+1}$ be a complete non-compact hypersurface $M$ in a space form $Q_c^{m+1}$, with $c\leq 0$. Let $P_j=P_j(A)$, $j=1,2,\ldots$, be the Newton operators associated to the shape operator $A$ of the immersion $f$. Assume that $$S_{j+1}= 0 \ \mbox{ and } \  S_j \mbox{ does not change of sign,}$$ for some $1\leq j\leq m-1$. Then either the rank $\rk(A(q))\leq j-1$ or the rate of growth of the integral $\int_M |S_j|$ is at least linear with respect to the geodesic balls of $M$ centered at $q$.
\end{theorem}

Our next theorem says the following:

\begin{theorem}\label{thm_majorar} Let $f:M^m\to \bar M$ be an isometric immersion of a complete non-compact manifold $M$ in a complete simply-connected manifold $\bar M$. Assume that the radial curvature of $\bar M$ with respect to some base point $q_0\in f(M)$ satisfies $\bar K_{\rad}\leq -c^2$, for some constant $c\geq 0$. Let $\Phi:TM\to TM$ be a positive-semidefinite symmetric operator of class $C^1$ such that $\tr\Phi(q_0)>0$.  Assume that  
\begin{equation*}
|H_\Phi + \dv \Phi|\leq \frac{(m-1)c}{m} \, \tr\Phi
\end{equation*} 
Assume further that $|\Phi \na r|\leq \frac{\tr\Phi}{m}$, where $r=d_{\bar M}(\cdot\,,q_0)$ and $\na r=(\bar\na r)^T$ denotes the orthogonal projection of the gradient vector $\bar \na r$ to $TM$. Then it holds that 
\begin{equation}\label{eq-thm-majorar}
\liminf_{\mu\to \infty} \frac{\mu^{-1}}{\tr\Phi(q_0)}\int_{B_{\mu(q_0)}}\tr\Phi >0. 
\end{equation}
Moreover the limit in (\ref{eq-thm-majorar}) does not depend of $q_0$. In particular, the integral $\int_M \tr\Phi$ is infinite.
\end{theorem}

The result bellow follows from Theorem \ref{thm_majorar} by considering $\Phi=\la^s I$, with $1\leq s\leq p$, where $\la:M\to \real$ is a  nonnegative $C^1$-function.

\begin{corollary}\label{majorada} Let $f:M^m\to \bar M$ be an isometric immersion of a complete non-compact manifold $M$ in a Hadamard manifold $\bar M$ with sectional curvature satisfying $\bar K\leq -c^2$, for some constant $c\geq 0$. Let $\la:M\to [0,\infty)$ be a nonnegative $C^1$ function satisfying:
\begin{equation*}
|\la H + p\na\la|\leq {(m-1)c}\la,
\end{equation*}
for some constant $p\geq 1$, where $H=\tr\II$ is  the mean curvature vector field of the immersion $f$. Then, for all $q\in M$, either $\la(q)=0$ or the rate of growth of the integral $\int_M \la^s$ satisfies
\begin{equation}\label{eq-majorada}
\liminf_{\mu\to\infty} \frac{\mu^{-1}}{\la^s(q)}\int_{B_\mu(q)}\la^s>0,
\end{equation}
for all $1\leq s\leq p$. Moreover the limit in (\ref{eq-majorada}) does not depend of $q$.
In particular, either $\la\equiv 0$ or the integral $\int_M\la^s$ is infinite, for all $1\leq s\leq p$.
\end{corollary}

Let $M$ be a Hadamard manifold with sectional curvature satisfying $K\leq -c^2$, for some constant $c\geq 0$. It is simple to show that the warped product space $\bar M=\real\times_{\cosh(ct)}M$ is also a Hadamard manifold with sectional curvature satisfying $\bar K\leq -c^2$. Furthermore the inclusion map $i:M\to \{0\}\times M\subset \bar M$ is a totally geodesic embedding. Thus it follows from Corollary \ref{majorada} the result below:

\begin{corollary} Let $M$ be a Hadamard manifold with sectional curvature satisfying $K\leq -c^2$, for some constant $c\geq 0$. Let $\la:M\to \real$ be a nonnegative $C^1$ function satisfying $|\na \la|\leq \frac{(m-1)c}{p} \la$, for some constant $p\geq 1$. Then either $\la \equiv 0$ or the rate of growth of $\int_M \la^s$ satisfies 
\begin{equation}\label{eq-majorada-2}
\liminf_{\mu\to\infty} \frac{\mu^{-1}}{\la^s(q)}\int_{B_\mu(q)}\la^s>0,
\end{equation}
for all $1\leq s\leq p$. Moreover the limit in (\ref{eq-majorada-2}) does not depend of $q$.
In particular, either $\la\equiv 0$ or the integral $\int_M\la^s$ is infinite, for all $1\leq s\leq p$.
\end{corollary}

Finally we will enunciate our last theorem. We recall that an end $E$ of $M$ is an unbounded connected component of the complement set $M-\Om$, for some compact subset $\Om$ of $M$. We say that a manifold $\bar M$ has bounded geometry if it has sectional curvature bounded from above and injective radius bounded from below by a positive constant. By results of Frensel \cite{fr},  Cheng, Cheung and Zhou \cite{ccz} and do Carmo, Wang and Xia \cite{cwx} we know the following theorem.
\begin{theoremB}
Let $f:M\to \bar M$ be an isometric immersion of a complete non-compact manifold $M$ in a manifold $\bar M$ with bounded geometry. If the mean curvature vector of $f$ is bounded in norm then each end of $M$ has infinite volume.
\end{theoremB}
Actually, Cheng, Cheung and Zhou \cite{ccz} improve Theorem B by showing that the volume growth of each end $E$ of $M$ is at least linear, that is 
\begin{equation}\label{eq-ccz}
\liminf_{\mu\to \infty} \frac{\vol(B_\mu(q)\cap E)}{\mu}>0,
\end{equation}
for all $q\in E$. Moreover the limit (\ref{eq-ccz}) does not depend of of $q$ (see Proposition 2.1 of \cite{ccz}).

Our last theorem says the following:

\begin{theorem}\label{bound_geom} Let $f:M^m\to \bar M$ be an isometric immersion of a complete non-compact manifold $M$ in a manifold $\bar M$ with bounded geometry. Let $E$ be an end of $M$ and $\la:E\to \real$ a nonnegative $C^1$ function. Assume that the mean curvature vector field $H=\tr\II$ of the immersion $f$ satisfies $$|H\la+p\na\la|\leq \ka\, \la \ \mbox{ in } E,$$ for some constant $\ka\geq 0$. Then it holds that $\lim_{x\to \infty \atop{x\in E}} \la(x)=0$ or the integrals $\int_E \la^s$ are infinite, for all $1\leq s\leq p$.
\end{theorem}

Note that if $M$ has bounded geometry then the product manifold $\bar M=M\times \real$ also has bounded geometry. Since the inclusion map $i:M\to M\times \{0\}\subset \bar M$ is a totally geodesic embedding, the result below follows from Theorem  \ref{bound_geom}.

\begin{corollary} Let $E$ be an end of a complete non-compact manifold with bounded geometry. Let $\la:E\to [0,\infty)$ be a nonnegative $C^1$ function. Assume that $$|\na\la|\leq \ka \la \mbox{ in E,}$$ for some constant $\ka\geq 0$. Then it holds that $\lim_{x\to \infty \atop{x\in E}} \la(x)=0$ or the integrals $\int_E \la^p$ are infinite, for all $p\geq 1$.
\end{corollary}


\section{Preliminaries}\label{preliminaries}
Let $f:M^m\to \bar M$ be an isometric immersion of an $m$-dimensional Riemannian manifold $M$ in a Riemannian manifold $\bar M$. For the sake of simplicity, henceforth we shall make the usual identification of the points $f(q)$ with $q$ and the vectors $df_qv$ with $v$, for all $q\in M$ and $v\in T_qM$. 
Let $\Phi:TM\to TM$ be a symmetric operator of class $C^1$. We consider the following definition:
\begin{definition} The {\it $\Phi$-divergence} of a vector field $X:M\to T\bar M$ of class $C^1$ is given by the following: 
\begin{equation*}
\m D_{\Phi} X =\tr\left\{Z\in TM\mapsto \Phi \left(\bar\nabla_Z X\right)^T\right\},
\end{equation*}
Note that if $\Phi=I:TM\to TM$ is the identity operator then $D_{\Phi}X$ coincides with the divergence $\rm{div}_MX$. 
\end{definition}

Let $u=(u^1,\ldots,u^m)$ be a local coordinate system on $M$. Let $\{\frac{\p}{\p u^1},\ldots,\frac{\p}{\p u^m}\}$ and  $\{du^1,\ldots,du^m\}$ be the frame and coframe associated to $u$, respectively. Using the Einstein's summation convention, let $\lan\,,\ran=g_{ij}du^i \otimes du^j$ be the metric of $M$. The Cheng-Yau square operator \cite{cy} associated to the symmetric $(0,2)$-tensor $\phi=\phi_{ij}du^i\otimes du^j$, where $\phi_{ij}=\lan \Phi(\frac{\p}{\p u^i}),\frac{\p}{\p u^j}\ran=\Phi^k_ig_{kj}$, is defined by
\begin{equation}\label{chen-yau}
\square_{\phi} f = \m D_{\Phi} (\na f).
\end{equation}
It was proved in \cite{cy} that the operator $\square_{\phi}$ is self-adjoint on the space of the all Sobolev functions with null trace if, and only if, the covariant derivative of $\Phi$ satisfies $\Phi^i_{j,i}=0$, for all $j$. Let $X:M\to T\bar M$ be a vector field of class $C^1$ and write $X^T=X^j\frac{\p}{\p u^j}$. Using that  $\na_{\frac{\p}{\p u^i}}X^T=\na_{\frac{\p}{\p u^i}}(X^j\frac{\p}{\p u^j})=X^j_{,i}\frac{\p}{\p u^j}$, we have that
\begin{equation*}
\Phi(\na_{\frac{\p}{\p u^i}}X^T)=\Phi(X^j_{,i}\frac{\p}{\p u^j})=X^j_{,i}\Phi^k_j\frac{\p}{\p u^k}.
\end{equation*} 
Thus we have that $\m D_{\Phi}(X^T)=X^j_{,i}\Phi^i_j$. Since $\Phi(X^T)=X^j\Phi(\frac{\p}{\p u^j})=X^j\Phi_j^i\frac{\p}{\p u^i}$, we obtain
\begin{eqnarray}\label{divergent}
\dv_M(\Phi (X^T))&=& (X^j\Phi_j^i)_{,i}=X^j_{,i}\Phi^i_j+X^j\Phi^i_{j,i}\\&=&\m D_{\Phi}X^T + (\dv\Phi)^*(X^T),\nonumber 
\end{eqnarray}
where $(\dv\Phi)^*$ the $1$-form defined by $(\dv\Phi)^*=\Phi^i_{j,i}du^j$. Using that $(\na_{\frac{\p}{\p u^i}}\Phi)\frac{\p}{\p u^j}=\Phi_{j,i}^k\frac{\p}{\p u^k}$ we have that $(\na_{\frac{\p}{\p u^i}}\Phi)X^T=X^j(\na_{\frac{\p}{\p u^i}}\Phi)\frac{\p}{\p u^j}=X^j\Phi_{j,i}^k\frac{\p}{\p u^k}$. This implies that 
\begin{eqnarray}\label{div}
(\dv\Phi)^*(X^T)&=& X^j\Phi_{j,i}^i=\tr\left\{Y\in TM \mapsto (\na_{Y}\Phi)X^T\right\}\\&=&\lan \dv \Phi, X^T\ran.\nonumber
\end{eqnarray}

\begin{proposition}\label{phi vector} Let $X:M^m\to T\bar M$ be a $C^1$ vector field and $f:M\to \real$  a $C^1$ function. Then the following statements hold:
\begin{enumerate}[(A)]
\item\label{phi tangent} $\m D_{\Phi} X = \m D_{\Phi} X^T - \lan H_{\Phi},X\ran$; 
\item\label{phi function} $\m D_{\Phi} (fX)=f\m D_{\Phi}X + \lan \Phi (X^T),\nabla f\ran$;
\item\label{divphi} $\m D_{\Phi} X= \dv_M(\Phi(X^T)) - \lan (H_{\Phi}+\dv_M\Phi) , X\ran. 
$
\end{enumerate}
\end{proposition}
\begin{proof}
Write $X=X^T+X^N$, where $X^N$ is the orthogonal projection of $X$ to $T^\perp M$. Let $\{e_1,\ldots,e_m\}$ be a local orthonormal frame of $M$. We have that
\begin{eqnarray}
\m D_{\Phi} X &=& \sum_{i=1}^m \lan \Phi (\bar\nabla_{e_i}X)^T,e_i\ran = \sum_{i=1}^m \lan \Phi (\nabla_{e_i}X^T),e_i\ran + \sum_{i=1}^m \lan \Phi (\bar\nabla_{e_i}X^N)^T , e_i\ran \nonumber \\ &=& \m D_{\Phi} X^T - \sum_{i=1}^m \lan \II(e_i,\Phi e_i), X^N\ran= \m D_{\Phi} X^T-\lan H_{\Phi},X\ran,\nonumber
\end{eqnarray}
which proves Item \ref{phi tangent}.  Now,
\begin{eqnarray*}
\m D_{\Phi} (fX)&=& \sum_{i=1}^m \lan \Phi (\bar\nabla_{e_i}fX)^T,e_i\ran= \sum_{i=1}^m \lan e_i(f) \Phi (X^T)+f\Phi (\bar\nabla_{e_i}X)^T,e_i\ran\nonumber\\ &=&f\m D_{\Phi}X + \sum_{i=1}^m e_i(f) \lan \Phi (X^T), e_i\ran= f\m D_{\Phi}X +\lan \Phi (X^T),\nabla f\ran,
\end{eqnarray*}
which proves Item \ref{phi function}.
Using \ref{phi tangent}, (\ref{divergent}) and (\ref{div}) we obtain that
\begin{equation}\label{D phi}
\m D_{\Phi} X= \dv_M(\Phi X^T) - \lan (H_{\Phi}+\dv_M\Phi) , X\ran,  
\end{equation}
which proves Item \ref{divphi}.
\end{proof}
\begin{remark}
By \ref{divphi} and the divergence theorem we have that
\begin{equation}\label{divergence}
\int_D \m D_{\Phi} X =  \int_{\p D} \lan \Phi X^T,\nu \ran  - \int_D \lan (H_{\Phi} + \dv_M\Phi), X \ran,
\end{equation}
where $D$ is a bounded domain with Lipschitz boundary $\p D$ and $\nu$ is the exterior conormal along $\p D$. Equation (\ref{divergence}) holds in the sense of the trace.
\end{remark}

\begin{example}\label{radialvector}
Take $p\in M$ and let $x=(x^1,\ldots,x^n)$ be a coordinate system in a normal neighborhood $V$ of $p$ in $\bar M$. Write $\frac{\p^T}{\p x^i}=(\frac{\p}{\p x^i})^T=a_i^j\frac{\p}{\p x^j}$ and $\Phi(\frac{\p^T}{\p x^i})=\bar\phi_i^j\frac{\p}{\p x^j}$. We have that 
\begin{equation*}
\Phi\left(\bar\na_{\frac{\p^T}{\p x^i}}\frac{\p}{\p x^j}\right)^T=a_i^k\bar \Ga_{kj}^l \Phi(\frac{\p^T}{\p x^l})=a_i^k\bar \Ga_{kj}^l \bar\phi_l^s \frac{\p}{\p x^s},
\end{equation*} 
where $\bar\Ga_{ij}^k$ are the Christoffel symbols of the Riemannian connection $\bar\na$ of $\bar M$.
This implies that $\m D_{\Phi}(\frac{\p}{\p x^j})=a_i^k\bar \Ga_{kj}^l \bar\phi_l^i$. Using \ref{phi function}, it holds the following: 
\begin{equation*}
\m D_{\Phi}(h\frac{\p}{\p x^j})=h a_i^k\bar \Ga_{kj}^l \bar \phi_l^i + \bar\phi_j^k \frac{\p h}{\p x^k},
\end{equation*}
for all $h\in C^1(M)$.
As a particular case, consider the vector field $Y=r\bar\na r$, where $r:\bar M\to [0,\infty)$ is the distance function $r(q)=d_{\bar M}(q,q_0)$, for some $q_0\in \bar M$. Using that $r^2(q)=(x^1(q))^2+\ldots+(x^n(q))^2$, for all $q\in V$, we have that $Y(q)=\frac{1}{2}\bar\na r^2=x^j(q)\frac{\p}{\p x^j}$. Using that $Y= x^j\frac{\p}{\p x^j}$ and $\tr_M\Phi=\bar\phi^i_i$, we obtain that
\begin{equation*}
\m D_{\Phi}(Y)=x^j a_i^k\bar \Ga_{kj}^l \bar\phi_l^i + \bar\phi_j^k \frac{\p x^j}{\p x^k}=\tr_M \Phi + x^j a_i^k\bar \Ga_{kj}^l \bar\phi_l^i.
\end{equation*}
In particular, if $\bar M$ is flat then $\m D_{\Phi}(Y)=\tr\Phi$. Since $Y=2^{-1}\bar\na r^2$, using \ref{phi tangent} and (\ref{chen-yau}), we obtain that  
\begin{eqnarray*}
2^{-1}\square_{\Phi} r^2 &=& \m D_{\Phi} Y+ r \lan H_{\Phi},\bar\na r\ran\\ &=& \tr\Phi + r\lan H_{\Phi},\bar\na r\ran + x^j a_i^k\bar \Ga_{kj}^l \bar\phi_l^i.\nonumber
\end{eqnarray*}
\end{example}

Let $\m K:\real \to \real$ be an even continuous function. Let $h$ be a solution of the Cauchy problem
\begin{equation}\label{cauchy-K}\left\{
\begin{array}{l}
h''+\m Kh=0,\\
h(0)=0, h'(0)=1.
\end{array}
\right.
\end{equation}
Let $I=(0,r_0)$ be the maximal interval where $h$ is positive. 

Let $\bar M$ be a Riemannian manifold and $\m B=\m B_{r_0}(q_0)$ the geodesic ball of $\bar M$ with center $q_0$ and radius $r_0>0$. 
Consider the radial vector field $$X=h(r)\bar \na r,$$ defined on $\m B\cap V$, where $r=d_{\bar M}(\cdot\,,q_0)$, for some fixed point $q_0\in M$, and $V$ is a normal neighborhood of $q_0$ in $\bar M$. The result below generalizes Example \ref{radialvector} and will be useful is the proof of Theorem \ref{thm_asz}.
\begin{proposition}\label{hessiana-K} Let $f:M^m\to \bar M$ be an isometric immersion of a manifold $M$ in the manifold $\bar M$. Assume that the radial curvature of $\bar M$ with respect to the basis point $q_0\in M$ satisfies $\bar K_{\rad}\leq \m K(r)$ in $\m B\cap V$. Let $\Phi:TM\to TM$ be a symmetric operator.  Assume that one of the following conditions hold:
\begin{enumerate}[(i)]
\item\label{positive} $\Phi$ is positive-semidefinite; or
\item\label{curvature} $\bar K_{\rad}= \m K(r)$,
\end{enumerate}
Then, it holds that
\begin{equation}\label{prop_comparison}
\m D_{\Phi} X\geq h'(r)\tr\Phi.
\end{equation} 
Moreover, the equality occur if \ref{curvature} holds.
\end{proposition}

\begin{proof} 
Take $q\in M\cap V$. Let $\xi=\{e_1,\ldots,e_m\}$ be an orthonormal basis of $T_qM$ by eigenvectors of $\Phi$ and $\{\la_1,\ldots,\la_m\}$ the corresponding eigenvalues. We extend the basis $\xi$ to an orthonormal frame on a neighborhood of $q$ in $M$. Then, at the point $q$, we have
\begin{equation*} 
\m D_{\Phi} \bar \na r=\sum_{i=1}^m \lan \Phi(\na_{e_i} \bar\na r)^T,e_i\ran=\sum_{i=1}^m \la_i \lan \na_{e_i} \bar\na r, e_i\ran=\sum_{i=1}^m \la_i\,(\hs_{\bar M}r)_q(e_i,e_i),
\end{equation*}
where $\hs_{\bar M}r$ is the Hessian of $r$. 

Let $\ti M$ be the Euclidean ball of $\real^m$ of radius $r_0$ and center at the origin $0$ endowed with the metric $\ti g$, which in polar coordinates can be written as 
\begin{equation*}
\ti g= d\rho^2+h(\rho)^2d\om^2,   
\end{equation*}
where $d\om^2$ represents the standard metric on the Euclidean unit sphere $S^{m-1}$. Consider the distance function $\ti r=d_{\ti M}(\cdot\,,0)$. Then, for $x=\rho\om$ with $\rho>0$ and $\om\in S^{m-1}$, the Hessian of $\ti r$ satisfies:
\begin{equation*}
\hs(\ti r)=\frac{h'}{h}(\ti g-d\ti r\otimes d\ti r),
\end{equation*}
Furthermore, the function $\m K(\ti r)=-\frac{h''(\ti r)}{h(\ti r)}$ is the radial curvature of $\ti M$ with respect to the basis point $0$. Thus since the radial curvature of $\bar M$ with respect to some basis point $q_0$ satisfies $\bar K_{\rad}\leq \m K(r)$ it follows from the Hessian comparison theorem (see Theorem A page 19 of \cite{gw}) the following:
\begin{equation}\label{comparison}
(\hs_{\bar M}r)(e_i,e_i)\geq \frac{h'(r)}{h(r)}\left(1-\lan \na r, e_i\ran^2\right).
\end{equation}
Moreover, the equality in (\ref{comparison}) holds when $\bar K_{\rad}=\m K(r)$. Since $\lan \Phi e_i,e_j\ran_q=\la_i\de_{ij}$, for all $i,j$, we obtain  
\begin{eqnarray}\label{compar}
\m D_{\Phi}X&=& \m D_{\Phi} (h(r)\bar \na r)=h(r)\m D_{\Phi}\bar\na r+ h'(r)\lan \na r,\Phi\na r\ran\\ &=& h(r)\sum_{i=1}^m \la_i\,(\hs_{\bar M}r)_q(e_i,e_i)+ h'(r)\lan \na r,\Phi\na r\ran.\nonumber
\end{eqnarray}
Using that the hypothesis \ref{positive} and \ref{curvature}, it follows from (\ref{comparison}) and (\ref{compar}) that
\begin{eqnarray}\label{compar2}
\m D_{\Phi}X &\geq& h(r)\frac{h'(r)}{h(r)} \sum_{j=1}^m \la_j \left(1-\lan\na r, e_j\ran^2\right) + h'(r)\lan \na r,\Phi\na r\ran\\&=& h'(r)\left(\tr\Phi-\lan\na r, \Phi\na r \ran\right) + h'(r)\lan \na r,\Phi\na r\ran\nonumber\\&=&h'(r)\tr\Phi.\nonumber
\end{eqnarray}
Moreover, the equality in (\ref{compar2}) holds when $\bar K_{\rad}=\m K(r)$.
\end{proof}
\begin{corollary}\label{theodiv} Under the hypotheses of Proposition \ref{hessiana-K} we have that
\begin{equation*} 
\int_{\p D} h(r)\lan\Phi \na r,\nu \ran \geq \int_D \left(h'(r)\tr\Phi - h(r)|H_\Phi+\dv_M\Phi|\right),
\end{equation*}
where $D$ is a bounded domain compactly contained in $V\cap \m B$ with Lipschitz boundary $\p D$, and $\nu$ is the exterior conormal along $\p D$.
\end{corollary}
\begin{proof} Using \ref{divphi} we have that $\dv_M(\Phi X^T)=\m D_{\Phi}X + \lan H_{\Phi}+\dv\Phi, X\ran$. Since $|\bar\na r|=1$, 
using (\ref{prop_comparison}) and the divergence theorem (see for instance \cite{fe}), we obtain that
\begin{eqnarray}
\int_{\p D} h(r)\lan\Phi \na r,\nu \ran &\geq& \int_D h'(r)\tr\Phi + \int_D h(r)\lan H_{\Phi}+\dv\Phi,\bar\na r\ran\\&\geq& \int_D \left(h'(r)\tr\Phi - h(r)|H_\Phi+\dv_M\Phi|\right).\nonumber
\end{eqnarray}
Corollary \ref{theodiv} is proved.
\end{proof}
We denote by $\bar R_{q_0}$ the injectivity radius of $\bar M$ at the point $q_0$ and $B_\mu(q)$ the geodesic ball of $M$ with center $q$ and radius $\mu$. Let $\al:[0,\infty)\to [0,\infty)$ be a nonnegative $C^1$ function. We consider the following positive number: 
\begin{equation}\label{def-mu-al}
\mu_{\m K,\al}=\sup\,\Big\{ t\in (0, r_0);  \  \frac{h'(t)}{h(t)}>\al(t) \  \mbox{ and } \  \al'(t) \geq -\frac{h'(t)^2}{h(t)^2} - \m K(t)\Big\}.
\end{equation}


\section{proof of Theorem \ref{thm_asz}, Theorem \ref{corol-dist} and Corollary \ref{corol-dist}.}\label{sec-the-al}

The main tool of this section is the following result:

\begin{theorem}\label{theo-al} Let $f:M\to \bar M$ be an isometric immersion of a complete noncompact manifold $M$ in a manifold $\bar M$. Assume that the radial curvature of $\bar M$ with base point in some $q_0\in f(M)$ satisfies $\bar K_{\rad}\leq \m K(r)$, where $r=d_{\bar M}(\cdot\,,q_0)$ and $\m K:\real\to \real$ is an even continuous function. Let $\Phi:TM\to TM$ be a positive-semidefinite symmetric operator such that $\tr\Phi(q_0)>0$. Assume further that 
\begin{equation}\label{hip-theo-alp}
|H_\Phi + \dv \Phi|\leq \al(r)\, \tr\Phi,
\end{equation} 
where $\al:[0,\infty)\to [0,\infty)$ is a nonnegative $C^1$ function. Then, for each $0<\mu_0<\min\{\mu_{\m K,\al},\bar R_{q_0}\}$, there exists a positive constant $\La=\La(q_0,\mu_0,M)$ satisfying 
\begin{equation*}
\int_{B_\mu(q_0)} \tr\Phi \geq \La \int_{\mu_0}^\mu e^{-\int_{\mu_0}^\tau \al(s)ds}d\tau,
\end{equation*} 
for all $\mu_0\leq \mu<\min\{\mu_{\m K,\al}, \bar R_{q_0}\}$.
\end{theorem}

\begin{proof} Take $0<\mu<\bar R_0=\min\{\mu_{\m K,\al}, \bar R_{q_0}\}$ and let $B_\mu=B_\mu(q_0)$. Note that the distance function $\rho=d_M(\cdot\,,q_0)$ satisfies $r\leq \rho$. This implies that $B_{\mu}$ is contained in the geodesic ball $\m B_{\bar R_0}(q_0)$ of $\bar M$ with center $q_0$ and radius $\bar R_0$. Since $M$ is a complete noncompact manifold and $\rho$ is a Lipschitz function we obtain that the ball $B_\mu$ is a bounded domain of $M$ with Lipchitz boundary $\p B_\mu\neq \emptyset$. Since $|\na\rho|=1$ a.e. in $B_\mu$, using Corollary \ref{theodiv}, equation (\ref{hip-theo-alp}) and the coarea formula (see for instance \cite{fe} or Theorem 3.1 of \cite{fe2}), we obtain that 
\begin{eqnarray}\label{coarea1}
\int_{\p B_\mu} h(r)\lan\Phi\na r,\nu\ran &\geq& \int_{B_\mu} \Big(\frac{h'(r)-\al(r) h(r)}{h(r)}\Big)h(r)\tr\Phi\\ &=& \int_0^\mu \int_{\p B_{\tau}} \Big(\frac{h'(r)}{h(r)}-\al(r)\Big)h(r)\tr\Phi\nonumber,
\end{eqnarray}
for almost everywhere $0<\mu<\bar R_0$, where $\nu$ is the exterior conormal along $\p B_{\mu}$. 

Take $q\in M$ and let $\{e_1,\ldots,e_m\}\subset T_qM$ be an orthonormal basis by eigenvectors of $\Phi$ at the point $q$. Consider $\{\la_1,\ldots,\la_m\}$ the corresponding eigenvalues. Since $\Phi$ is positive-semidefinite  we have that $\la_i\geq 0$, for all $i$. Since $|\na r|\leq 1$, using Cauchy-Schwartz inequality, we obtain
\begin{equation}\label{cauchy}
\lan \Phi\na r,\nu\ran=\sum_{i=1}^m \la_i\lan \na r, e_i\ran \lan \nu, e_i\ran \leq \sum_{i=1}^m \la_i |\na r||\nu|=(\tr\Phi) |\na r|\leq \tr\Phi.
\end{equation}
Thus, it follows from (\ref{coarea1}) that 
\begin{equation}\label{decreas}
\int_{\p B_\mu} h(r)\tr\Phi\geq \int_0^\mu \int_{\p B_{\tau}}\Big(\frac{h'(r)}{h(r)}-\al(r)\Big) h(r)\tr\Phi,
\end{equation} 
for a.e.  $0<\mu<\bar R_0$. 

We define the following functions
\begin{equation}\label{def-FG}
\begin{array}{l}
F: \mu\in (0,\bar R_0)\mapsto F(\mu)=\int_{\p B_\mu} h(r)\tr\Phi;\\

G:\mu \in (0,\bar R_0)\mapsto G(\mu)=\int_0^\mu \int_{\p B_{\tau}}\Big(\frac{h'(r)}{h(r)}-\al(r)\Big) h(r)\tr\Phi,
\end{array}
\end{equation}
It follows from (\ref{decreas}) that 
\begin{equation}\label{F-geq-G}
F(\mu)\geq G(\mu), 
\end{equation}
for a.e. $0<\mu<\bar R_0$.

Note that $\al'(t) \geq -\frac{h'(t)^2}{h(t)^2} - \m K(t)$ is equivalent to say that $\Big(\frac{h'(t)}{h(t)}-\al(t)\Big)'\leq 0$. Thus, by hypothesis, the function $t\in (0,\mu_{\m K,\al})\mapsto \frac{h'(t)}{h(t)}-\al(t)$ is positive and non-increasing. Since the function $r=d_{\bar M}(\,\cdot\,,q_0)$ satisfies $r\leq \tau$ in $\p B_\tau$ we have that $\frac{h'(r)}{h(r)}-\al(r)\geq \frac{h'(\tau)}{h(\tau)}-\al(\tau)>0$ in $\p B_\tau$, for all $0<\tau\leq \bar R_0$. This implies that $G(\mu)>0$, for all $0<\mu<\bar R_0$, since $G\geq 0$, $G$ is nondecreasing and $G(\mu)>0$, for all $\mu>0$ sufficiently small (recall that $\tr\Phi(q_0)>0$).

Using (\ref{def-FG}) and (\ref{F-geq-G}), we have that
\begin{equation*}
G'(\mu)= \left(\frac{h'(\mu)}{h(\mu)}-\al(\mu)\right)F(\mu)\geq \left(\frac{h'(\mu)}{h(\mu)}-\al(\mu)\right)G(\mu),
\end{equation*}
for a.e. $0<\mu<\bar R_0$.
Thus we obtain
\begin{equation}\label{deriv-ln}
\frac{d}{d\mu} \ln G(\mu)= \frac{G'(\mu)}{G(\mu)}\geq \frac{h'(\mu)}{h(\mu)}-\al(\mu)=\Big(\frac{d}{d\mu}\ln h(\mu)\Big) -\al(\mu),
\end{equation}
for a.e. $0<\mu<\bar R_0$.
Integrating (\ref{deriv-ln}) over $(\mu_0,\mu)$, with $0<\mu_0<\mu$, we obtain that 
\begin{equation*}\label{ln}
\ln \frac{G(\mu)}{G(\mu_0)}\geq \ln\frac{h(\mu)}{h(\mu_0)} - \int_{\mu_0}^\mu \al(s)ds.
\end{equation*}
This implies that
\begin{equation}\label{equ-G}
G(\mu)\geq \La(q_0,\mu_0,M) h(\mu) e^{-\int_{\mu_0}^\mu \al(s)ds},
\end{equation}
where $\La=\La(\mu_0,q_0,M)=\frac{G(\mu_0)}{h(\mu_0)}$.

Now, we define the function $f(\mu)=\int_{B_\mu}\tr\Phi$, with $0<\mu<\bar R_0$. Since $h(r)\leq h(\mu)$ in $\p B_\mu$ it follows from (\ref{F-geq-G}), (\ref{equ-G}) and the coarea formula that
\begin{equation*}
f'(\mu)=\int_{\p B_\mu} \tr\Phi \geq \frac{1}{h(\mu)}\int_{\p B_\mu} h(r)\tr\Phi = \frac{F(\mu)}{h(\mu)}\geq \frac{G(\mu)}{h(\mu)}\geq \La\, e^{-\int_{\mu_0}^\mu \al(s)ds},
\end{equation*} 
for a.e. $0<\mu<\bar R_0$.
Since $f(\mu_0)\geq 0$ we have that
\begin{equation*}
\int_{B_\mu}\tr\Phi=f(\mu)\geq  \La\int_{\mu_0}^\mu e^{-\int_{\mu_0}^\tau \al(s)ds}d\tau.
\end{equation*}
This concludes the proof of Theorem \ref{theo-al}.
\end{proof}

Now we are able to prove Theorem \ref{thm_asz}, Theorem \ref{corol-dist}, Corollary \ref{corol-dist} and Theorem \ref{cor-asz}.
\subsection{Proof of Theorem \ref{thm_asz}}\label{proof-teo-cor}
First we observe that the injectivity radius $\bar R_{q_0}=+\infty$, since $\bar M$ has nonpositive radial curvature with base point $q_0$.
We take the functions $\m K(t)=0$, with $t\in\real$, and $\al(t)=\frac{1}{t+\ep}$, with $t\ge 0$. The function $h(t)=t$, with $t>0$,  is the maximal positive solution of (\ref{cauchy-K}). Furthemore, we have that  $\mu_{\m K,\al}=\infty$. Since $\bar K_{\rad}\leq 0=\m K(r)$ and $\tr\Phi(q_0)>0$, Theorem \ref{theo-al} applies. Thus it holds that $$\int_{B_\mu}\tr\Phi\geq \La \int_{\mu_0}^{\mu} e^{-\int_{\mu_0}^{\tau}\frac{ds}{s+\ep}}d\mu=\La(\mu_0+\ep) \log\Big(\frac{\mu+\ep}{\mu_0+\ep}\Big),$$ for all $0<\mu_0<\mu$, where $\La$ is a positive constant depending only on $q_0$, $\mu_0$ and $M$. This implies that
\begin{equation*}
\liminf_{\mu\to \infty} \frac{1}{\log(\mu)}\int_{B_\mu(q_0)}\tr\Phi>0.
\end{equation*}
Theorem \ref{thm_asz} is proved.

\subsection{Proof of Theorem \ref{cor-asz-phi-min}.} Similarly as in the proof of Theorem \ref{thm_asz} we have that $\bar R_{q_0}$. Consider the function $\m K(t)=0$, with $t\in \real$, and $\al(t)=0$, with $t\geq 0$. We have that $\mu_{\m K,\al}=+\infty$ and Theorem \ref{theo-al} applies. Thus we obtain that 
$\int_{B_\mu(q_0)} \tr\Phi \geq \La (\mu-\mu_0),$
for all $0<\mu_0<\mu$, where $\La$ is a positive constant depending only on $q_0$, $\mu_0$ and $M$. This implies that
\begin{equation*}
\liminf_{\mu\to \infty} \frac{1}{\mu}\int_{B_\mu(q_0)}\tr\Phi\geq \La>0.
\end{equation*}
Theorem \ref{cor-asz-phi-min} is proved.

\subsection{Proof of Corollary \ref{corol-dist}}
Fix $q\in M$. Let $\{E_1,\ldots,E_m\}$ and $\{\bar E_1,\ldots, \bar E_k\}$ be orthonormal frames of $TM$ and $\m D$ defined in a neighborhood $U$ of $q$ in $M$, respectively. Since $P_{\m D}(v)=\lan v,\bar E_l\ran \bar E_l$, for all $v\in TU$, we obtain that
\begin{eqnarray}\label{equ-dist}
\dv P_{\m D}&=& \sum_{i=1}^m (\na_{E_i}P_{\m D})E_i=\sum_{i=1}^m \na_{E_i}(P_{\m D}(E_i))-P_{\m D}(\na_{E_i}E_i)\nonumber\\&=&\sum_{i=1}^m\sum_{l=1}^k \Big(E_i\big(\lan E_i,\bar E_l\ran\big)\Big)\bar E_l + \lan E_i,\bar E_l\ran \na_{E_i}\bar E_l - \lan\na_{E_i}E_i,\bar E_l\ran \bar E_l\nonumber\\ &=& \sum_{i=1}^m\sum_{l=1}^k \lan E_i,\na_{E_i}\bar E_l\ran \bar E_l + \lan E_i,\bar E_l\ran \na_{E_i}\bar E_l\nonumber\\ &=& \sum_{l=1}^k (\dv_M(\bar E_l))\bar E_l + \sum_{l=1}^k \na_{\sum_{i=1}^m\lan E_i,\bar E_l\ran E_i}\bar E_l\nonumber\\&=& \sum_{l=1}^k (\dv_M(\bar E_l))\bar E_l  + \sum_{l=1}^k \na_{\bar E_l} \bar E_l.
\end{eqnarray}

Since the distribution $\m D$ is integrable, there exists an embedded submanifold $S\subset M$ satisfying $q\in S$ and $T_xS=\m D(x)$, for all $x\in S$.  Let $\{\ti E_1,\ldots,\ti E_k\}$ be an orthonormal frame, defined in a small neighborhood $U$ of $q$ in $S$, that is geodesic at $q$ with respect to the connection of $S$, namely, 
\begin{equation}\label{ref-geod-S}
(\na^S_{\ti E_l}{\ti E_s})_q=P_{\m D}(\na_{\ti E_l}{\ti E_s})_q=0,
\end{equation}
for all $l,s=1,\ldots,k$.  Now, let $\{\ti E_{k+1},\ldots,\ti E_m\}$ be an orthonormal frame of the normal bundle $TS^{\perp}$  defined in a small neighborhood of $q$ in $S$, that we can also assume to be $U$. We extend the frame $\{\ti E_1,\ldots,\ti E_m\}$ to an orthonormal frame defined in a small tubular neighborhood $W$ of $U$ in $M$ by parallel transport along minimal geodesics from $U$ to the points of $W$.  In particular, it holds that $(\na_{\ti E_\be} \ti E_l)_x=0$, for all $x\in U$, $l=1,\ldots, k$ and $\be=k+1,\ldots,m$. This fact, togheter with (\ref{ref-geod-S}), imply that
\begin{equation}\label{eq-dist}
(\dv_M(\ti E_l))_q = \sum_{i=1}^k \lan (\na^S_{\ti E_i}\ti E_l)_q,\ti E_i(q)\ran + \sum_{\be=k+1}^m \lan (\na_{\ti E_\be}\ti E_l)_q,\ti E_\be(q)\ran=0,
\end{equation}
for all $l=1,\ldots,k$.
Thus, by (\ref{equ-dist}), (\ref{ref-geod-S}) and (\ref{eq-dist}) we obtain that
\begin{equation}\label{secff1}
(\dv P_{\m D})_q= \sum_{l=1}^k\sum_{\be=k+1}^m \lan (\na_{\ti E_l}{\ti E_l})_q, \ti E_\be(q)\ran \ti E_\be(q) = \sum_{l=1}^k \II_{M}^S (\ti E_l(q),\ti E_l(q)),
\end{equation}
where $\II_M^S$ is the second fundamental form of the submanifold $S$ in $M$. 

On the other hand, the second fundamental form $\ti\II$ of the restriction $f|_S:S\to \bar M$ is given by:
\begin{equation}\label{secff2}
\ti \II(v,v)=\II_M^S(v,v)+\II(v,v)=\II_M^S(v,v)+\II(P_{\m D}v,v),
\end{equation}
for all $v\in T_xS=\m D(x)$, with $x\in S$, where $\II$ denotes the second fundamental form of the immersion $f:M\to \bar M$. Thus, by (\ref{secff1}) and (\ref{secff2}), we have that  
\begin{equation}\label{minS}
\tr\ti\II=\dv P_{\m D} + H_{\m P}.
\end{equation}
By hypothesis the isometric immersion $f|_S:S\to \bar M$ is minimal. Thus, by (\ref{minS}), it holds that $\tr\ti\II=\dv P_{\m D} + H_{\m P}=0$. Since $\tr P_{\m D}=k\geq 1$ it follows from Theorem \ref{cor-asz-phi-min} that the rate of growth of the volume $\vol(M)=\frac{1}{k}\int_M \tr P_{\m D}$ is at least linear with respect to the geodesic balls centered to any point of $M$. Corollary \ref{corol-dist} is proved.

\section{Proof of Theorem \ref{cor-asz} and Theorem \ref{cor-asz2}} Before we prove Theorem \ref{cor-asz} we need some preliminaries. 
Let $W^m$ be an $m$-dimensional vector space and $T:W\to W$ a symmetric linear operator on $W$. Consider the Newton operators $P_j(T):W\to W$, $j=0,\ldots,m$, associated to $T$. It is easy to shows that each $P_j(T)$ is a symmetric linear operator with the same eigenvectors of $T$. Let $\{e_1,\ldots,e_m\}$ be an orthonormal basis of $W$ by eigenvectors of $T$ and $\{\la_1,\ldots,\la_m\}$ the corresponding eigenvalues. Let $W_j=\{e_j\}^{\perp}$, $j=1,\ldots,m$, be the orthogonal hyperplane to $e_j$ and consider $T_j=T|_{W_j}:W_j\to W_j$.  The two lemmas below were proved for the case that $T$ is the shape operator $A(p)$ associated to a hypersurface of a Riemannian manifold evaluated at some point $p$ (see Lemma 2.1 of \cite{bc} and Proposition 2.4 of \cite{asz}, respectively). The proof in the general case follows exactly the same steps. 

\begin{lemma} For each $1\leq j\leq m-1$, the following items hold:
\begin{enumerate}[(a)]
\item\label{eigPj} $P_j(T)e_k=S_j(T_k)e_k$, for each $1\leq k\leq m$;
\item\label{Pj} $\tr\,(P_j(T))=\sum_{k=1}^m S_j(T_k)=(m-j)S_j(T)$;
\item\label{TPj} $\tr\,(TP_j(T))=\sum_{k=1}^m \la_kS_j(T_k) = (j+1)S_{j+1}(T)$;
\item\label{T2Pj} $\tr\,(T^2P_j(T))=\sum_{k=1}^m \la_k^2S_j(T_k)=S_1(T)S_{j+1}(T)-(j+2)S_{j+2}(T)$.
\end{enumerate}
\end{lemma}

\begin{lemma}\label{Pj-semidefinite} Assume that $S_{j+1}(T)=0$, for some $1\leq j\leq n-1$. Then $P_j(T)$ is semidefinite.
\end{lemma}

We also need of the following lemmas:

\begin{lemma}\label{rkT} Assume that $S_{j-1}(T)=S_{j}(T)=0$, for some $2\leq j\leq m$. Then the rank of $T$ satisfies $\rk(T)\leq j-2$. 
\end{lemma}
\begin{proof} 
If $T=0$ then there is nothing to prove since $\rk(T)=0\leq m-2$. Thus we can assume that $T\neq 0$. We will prove Lemma \ref{rkT} by induction on $m=\dim W$. 
First we assume that $m=2$. Since
\begin{eqnarray*}
\|T\|^2:=\la_1^2+\la_2^2= (\la_1+\la_2)^2-2\la_1\la_2=S_1(T)^2-2S_2(T)=0
\end{eqnarray*}
it follows that $T=0$. 

Now we assume that Lemma \ref{rkT} is true for any symmetric operator $Q:V^k\to V^k$ defined on a $k$-dimensional vector space $V$, with $2\leq k\leq m-1$. Since $S_{j-1}(T)=S_j(T)=0$, for some $2\leq j\leq m$, it follows from Lemma \ref{Pj-semidefinite} that the operators $P_{j-2}(T)$ and $P_{j-1}(T)$ are semidefinite. Thus using that $$\tr(P_{j-1}(T))=(m-j+1) S_{j-1}(T)=0$$ it follows that $P_{j-1}(T)=0$. Furthermore, the operator $T^2 P_{j-2}(T)$ is also semidefinite with trace satisfying $$\tr(T^2P_{j-2}(T))= S_1(T)S_{j-1}(T) - jS_j(T)=0,$$ which implies that $T^2P_{j-2}=0$. Since $(T^2P_{j-2})e_k=\la_k^2S_{j-2}(T_k)e_k=0$ we obtain that $\la_k=0$ or $S_{j-2}(T_k)=0$. Thus, using that $S_{j-1}(T_k)=\lan P_{j-1}(T)e_k,e_k\ran=0$ and $\dim (W_k)=m-1$, we obtain by the induction assumption that $\la_k=0$ or $\rk(T_k)\leq j-3$.  Since $T\neq 0$ there exists some eigenvalue $\la_k\neq 0$. Thus we obtain that $\rk(T_k)\leq j-3$ which implies that $\rk(T)\leq j-2$.   
\end{proof}

\begin{lemma}\label{dvPj=0} Let $B:TM\to TM$ be a symmetric operator of class $C^1$ that satisfies the Codazzi equation. Then it holds that $\dv (P_j(B))=0$. 
\end{lemma}
\begin{proof} We denote by $P_j=P_j(B)$, with $j=1,\ldots,m$. Take $p\in M$ and let $\{E_1,\ldots,E_m\}$ be an orthonormal frame defined on an neighborhood $V$ of $p$ in $M$, geodesic at $p$. We have that
\begin{eqnarray}\label{div Pj}
\dv P_j &=& \sum_{i=1}^m (\na_{E_i}P_j)E_i = \sum_{i=1}^m (\na_{E_i}S_j I - BP_{j-1})E_i \nonumber\\ &=& \sum_{i=1}^m \big(E_i(S_j)E_i - \na_{E_i}(BP_{j-1})E_i\big)\nonumber\\ &=& \na(S_j) - \sum_{i=1}^m \big((\na_{E_i}B)P_{j-1}(E_i) + B(\na_{E_i}P_{j-1})E_i\big)\nonumber\\ &=& \na(S_j)-B(\dv P_{j-1}) - \sum_{i=1}^m (\na_{E_i}B)P_{j-1}(E_i).
\end{eqnarray}
Let $X$ be a $C^1$ vector field on $M$. Since $(\na_X B)$ is a symmetric operator and $(\na_{E_i}B)X=(\na_X B)E_i$, for all $i=1,\ldots,m$, we obtain that
\begin{eqnarray}\label{trPjXB}
\sum_{i=1}^m \lan (\na_{E_i}B)P_{j-1}(E_i),X\ran &=& \sum_{i=1}^m \lan P_{j-1}(E_i),(\na_XB)E_i\ran \\&=& \tr\big(P_{j-1}(\na_XB)\big).\nonumber
\end{eqnarray}
It was proved by Reilly \cite{re} (see Lemme A of \cite{re}) that $\tr\big(P_{j-1}(\na_XB)\big)=\lan\na(S_j),X\ran$.
Thus, using (\ref{trPjXB}), we obtain that 
\begin{equation}\label{naSj}
\sum_{i=1}^m (\na_{E_i}B)P_{j-1}(E_i) = \na (S_j).
\end{equation}
Using (\ref{div Pj}) and (\ref{naSj}), we obtain that
\begin{equation*}
(\dv P_j)_p= (\na S_j)(p) - B(\dv P_{j-1})_p - (\na S_j)(p)= -B(\dv P_{j-1})_p.
\end{equation*}
Since $P_0=I$ we obtain by recurrence that $\dv P_j=(-1)^j B^j(\dv I)=0$. This concludes the proof of Lemma \ref{dvPj=0}.
\end{proof}

Now, we are able to prove Theorem \ref{cor-asz} and Theorem \ref{cor-asz2}. 
\subsection{Proof of Theorem \ref{cor-asz}.}
Since $S_{j+1}(B)=0$ it follows from Lemma \ref{Pj-semidefinite} that the operator $P_j(B(p)):T_pM \to T_pM$ is semidefinite at each point $p\in M$. Since $S_j$ does not change of sign we obtain that $\Phi=\ep P_j$ is positive-semidefinite, for some constant $\ep\in \{-1,1\}$. Since $B$ satisfies the Codazzi equation it follows from Lemma \ref{dvPj=0} that $\dv\Phi=\ep\,\dv P_j=0$. Since $|H_\Phi+\dv\Phi|=|H_{P_j}|\leq \frac{1}{r+\ep}$, where $r$ is the distance function of $\bar M$ from $q_0$ and $\tr\Phi(q_0)=|\tr P_j(q_0)|=(m-j)|S_j(B(q_0))|>0$ we can apply Theorem \ref{thm_asz} to conclude that the rate of growth of $\int_M \tr \Phi=(m-j)\int_M |S_j(B)|$ is at least logarithmic with respect to the geodesic balls of $M$ centered at $q_0$. Theorem \ref{cor-asz} is proved.

\subsection{Proof of Theorem \ref{cor-asz2}.} 
Since $S_{j+1}=S_{j+1}(A)=0$ it follows from Lemma \ref{Pj-semidefinite} that the operator $P_j(A(p)):T_pM \to T_pM$ is semidefinite at each point $p\in M$. Since $S_j$ does not change of sign we obtain that $\Phi=\ep P_j$ is positive-semidefinite, for some constant $\ep\in \{-1,1\}$. Since the shape operator $A$ satisfies the Codazzi equation it follows from Lemma \ref{dvPj=0} that $\dv\Phi=\ep\,\dv P_j=0$. Since $|H_\Phi+\dv\Phi|=|H_{P_j}|=(j+1)S_{j+1}=0$ and $\tr\Phi(q_0)=|\tr P_j(q_0)|=(m-j)|S_j(A(q_0))|>0$ we can apply Theorem \ref{cor-asz-phi-min} to conclude that the rate of growth of the integral $\int_M \tr \Phi=(m-j)\int_M |S_j(A)|$ is at least linear with respect to the geodesic balls of $M$ centered at $q_0$. Theorem \ref{cor-asz2} is proved.


\section{Proof of Theorem \ref{thm_majorar}, Corollary \ref{majorada} and Theorem \ref{bound_geom}}\label{proof-thm-majorar} 
The main tool of this section is the following result:
\begin{theorem}\label{phi r} Let $f:M\to \bar M$ be an isometric immersion of a complete noncompact manifold $M$ in a manifold $\bar M$. Assume that the radial curvature of $\bar M$ with base point in some $q_0\in f(M)$ satisfies $\bar K_{\rad}\leq \m K(r)$, where $r=d_{\bar M}(\cdot\,,q_0)$ and $\m K:\real\to \real$ is an even continuous function. Let $\Phi:TM\to TM$ be a positive-semidefinite symmetric operator such that $\tr\Phi(q_0)>0$. Assume further that 
\begin{equation}\label{hip-theo-al}
|H_\Phi + \dv \Phi|\leq \al(r)\, \tr\Phi \ \ \mbox{ and } \ \ m|\Phi \na r|\leq \tr\Phi,
\end{equation}
where $\al:[0,\infty)\to [0,\infty)$ is a nonnegative $C^1$-function. Then 
\begin{equation*}
\int_{B_\mu(q_0)}\tr\Phi\geq  m\,\tr\Phi(q_0) \int_{0}^\mu h(\tau)^{m-1}e^{-m\int_0^\tau \al(s)ds}d\tau.
\end{equation*} 
for all $0<\mu<\min\{\mu_{\m K,\al}, \bar R_{q_0}\}$, where $h:(0,r_0)\to (0,\infty)$ is the maximal positive solution of (\ref{cauchy-K}).
\end{theorem}

\begin{proof}   
By following exactly the same steps as in the proof of Theorem \ref{theo-al} we obtain that 
\begin{equation}\label{decreasing}
\int_{\p B_\mu} h(r)\lan\Phi\na r,\nu\ran\geq \int_0^\mu \int_{\p B_{\tau}} \Big(\frac{h'(r)}{h(r)}-\al(r)\Big)  h(r)\tr\Phi,
\end{equation}
for almost everywhere $0<\mu<\bar R_0=\min\{\mu_{\m K,\al},\bar R(q_0)\}$, where $B_\mu=B_\mu(q_0)$ and $\nu$ is the exterior conormal along $\p D$.

Since $|\nu|=1$ and $|\Phi \na r|\leq \frac{\tr\Phi}{m}$, using Cauchy-Schwartz inequality, we obtain that $\lan\Phi\na r,\nu\ran\leq \frac{\tr\Phi}{m}$. 
Using (\ref{decreasing}) we obtain
\begin{equation}\label{inequality}
\int_{\p B_\mu} h(r) \tr\Phi \geq m\int_0^\mu \int_{\p B_{\tau}} \Big(\frac{h'(r)}{h(r)}-\al(r)\Big)h(r)\tr\Phi,
\end{equation}
for a.e.  $0<\mu<\bar R_0$.

Consider the following functions 
\begin{equation*}
\begin{array}{l}
F: \mu\in (0,\bar R_0)\mapsto F(\mu)=\int_{\p B_\mu} h(r)\tr\Phi\\

G:\mu \in (0,\bar R_0)\mapsto G(\mu)=\int_0^\mu \int_{\p B_{\tau}} \Big(\frac{h'(r)}{h(r)}-\al(r)\Big)h(r)\tr\Phi.
\end{array}
\end{equation*}
It follows by (\ref{inequality}) that 
\begin{equation}\label{F-geq-mG}
F(\mu)\geq m\, G(\mu),
\end{equation}
for a.e. $\mu\in (0,\bar R_0)$. 

Note that $G(\mu)>0$, for all $0<\mu<\bar R_0$, since $G\geq 0$, $G$ is nondecreasing and $G(\mu)>0$, for $\mu>0$ sufficiently small (recall that $\tr\Phi(q_0)>0$). Thus, by (\ref{F-geq-mG}), we obtain
\begin{equation*}
G'(\mu)= \Big(\frac{h'(\mu)}{h(\mu)}-\al(\mu)\Big)F(\mu)\geq m\Big(\frac{h'(\mu)}{h(\mu)}-\al(\mu)\Big)G(\mu),
\end{equation*}
for a.e.  $0<\mu<\bar R_0$.
This implies that
\begin{eqnarray}\label{der-ln}
\frac{d}{d\mu} \ln G(\mu)&=& \frac{G'(\mu)}{G(\mu)}\geq m\Big(\frac{h'(\mu)}{h(\mu)}-\al(\mu)\Big)=m\Big(\Big(\frac{d}{d\mu}\ln h(\mu)\Big) -\al(\mu)\Big)\nonumber\\&=& \Big(\frac{d}{d\mu}\ln h(\mu)^m \Big) -m\al(\mu),
\end{eqnarray}
for a.e. $\mu\in (0,\bar R_0)$. Integrating (\ref{der-ln}) over $(\mu_0,\mu)$, with $0<\mu_0<\mu$, we obtain that 
\begin{equation*}
\ln \Big(\frac{G(\mu)}{G(\mu_0)}\Big)\geq \ln \Big(\frac{h(\mu)^m}{h(\mu_0)^m}\Big) - m\int_{\mu_0}^\mu \al(s)ds.
\end{equation*}
Thus, we obtain
\begin{equation}\label{G}
G(\mu)\geq \frac{G(\mu_0)}{h(\mu_0)^m} h(\mu)^m e^{-m\int_{\mu_0}^\mu \al(s)ds} ,
\end{equation}
for all $0<\mu_0<\mu<\bar R_0$.  

Using that $r\leq \mu$ in $B_\mu$ and the function $\mu \in (0,\bar R_0)\mapsto \frac{h'(\mu)}{h(\mu)}-\al(\mu)$ is non-decreasing we have from the coarea formula that $$G(\mu)=\int_{B_\mu}\Big(\frac{h'(r)}{h(r)}-\al(r)\Big)h(r)\tr\Phi \geq \Big(\frac{h'(\mu)}{h(\mu)}-\al(\mu)\Big)\int_{B_\mu} \tr\Phi,$$
for all $0<\mu <\bar R_0$. 
Since $\lim_{t\to 0}\frac{h(t)}{t}=h'(0)=1$ and $h(0)=0$  we obtain  
\begin{eqnarray}\label{G-traco}
\lim_{\mu_0\to 0} \frac{G(\mu_0)}{h(\mu_0)^m}&\geq &\lim_{\mu_0\to 0} \Big(\frac{\mu_0}{h(\mu_0)}\Big)^m \lim_{\mu_0\to 0}\Big(\frac{1}{\mu_0^m}\int_{B_{\mu_0}}\big(h'(r)-\al(r)h(r)\big)\tr\Phi\Big)\nonumber\\&=&\big(h'(0)-\al(0)h(0)\big)\tr\Phi(q_0) = \tr\Phi(q_0).
\end{eqnarray}
Thus, using (\ref{G}), (\ref{G-traco}) and taking $\mu_0\to 0$, we obtain that
\begin{equation}\label{G-des}
G(\mu)\geq \tr\Phi(q_0)h(\mu)^m e^{-m\int_0^\mu \al(s)ds},
\end{equation}
for all $0<\mu<\bar R_0$.

Now we consider the function 
\begin{equation*}
\mu\in [0, \bar R_0) \mapsto f(\mu)=\int_{B_\mu}\tr\Phi.
\end{equation*}
Since $h(r)\leq h(\mu)$ in $\p B_\mu$ and $F(\mu)\geq m\,G(\mu)$, using the coarea formula and (\ref{G-des}), we obtain
\begin{eqnarray*}
f'(\mu)&=&\int_{\p B_\mu} \tr\Phi \geq \frac{1}{h(\mu)}\int_{\p B_\mu} h(r)\tr\Phi = \frac{F(\mu)}{h(\mu)}\geq m \frac{G(\mu)}{h(\mu)}\\&\geq&  m\,\tr\Phi(q_0) h(\mu)^{m-1}   e^{-m\int_{0}^\mu \al(s)ds}.
\end{eqnarray*}
Since $f(0)=0$, by integration $f'(\mu)$ on $(0,\mu)$, we have that
\begin{equation*}
\int_{B_\mu}\tr\Phi=f(\mu)\geq  m\, \tr\Phi(q_0) \int_{0}^\mu h(\tau)^{m-1}e^{-m\int_{0}^\tau \al(s)ds}d\mu.
\end{equation*}
This concludes the proof of Theorem {\ref{phi r}}.
\end{proof}

Now we are able to prove Theorem \ref{thm_majorar} and Theorem \ref{bound_geom}.

\subsection{Proof of Theorem \ref{thm_majorar}}
First we observe that the injectivity radius of $\bar M$ at the point $q_0$ satisfies $\bar R_{q_0}=+\infty$ since the radial curvature of $\bar M$ with base point $q_0$ is nonpositive. We consider constant functions 
\begin{equation*}
\m K(t)=-c^2 \ \mbox{ and } \ \al(t)=\frac{(m-1)c}{m},
\end{equation*}
 for all $t$. The maximal positive solution of (\ref{cauchy-K}) is given by $h(t)=\frac{1}{c}\sinh(c\,t)$, with $t>0$. Since $\cosh(t)\geq \sinh(t)$, for all $t\geq 0$, we obtain 
\begin{equation*}
\begin{array}{l}
h'(t)=\cosh(c\,t)> \frac{(m-1)c}{m}h(t) \ \mbox{ and }\\
0=\al'(t)\geq -c^2\big((\coth(c\,t))^2-1\big)=-\frac{h'(t)^2}{h(t)^2}-\m K(t),
\end{array}
\end{equation*}
for all $t>0$, which implies that $\mu_{\m K,\al}=\infty$. Thus, using Theorem \ref{phi r}, we obtain that 
\begin{eqnarray}\label{rate}
\int_{B_\mu}\tr\Phi &\geq& \frac{m}{c^{m-1}}\, \tr\Phi(q_0) \int_{0}^\mu \sinh(c\,\tau)^{m-1}e^{-(m-1)c\,\tau} d\tau\nonumber\\&=& \frac{m}{(2\,c)^{m-1}}\, \tr\Phi(q_0) \int_{0}^\mu (1 - e^{-2c\tau})^{m-1}d\tau\nonumber\\&\geq&
\frac{m}{(2\,c)^{m-1}}\, \tr\Phi(q_0) \int_{0}^\mu (1 - (m-1) e^{-2c\tau})d\tau.
\end{eqnarray}
The last inequality follows from the Bernoulli's inequality since $e^{-2c\tau}<1$. This implies that 
\begin{equation*}
\liminf_{\mu\to\infty} \frac{\mu^{-1}}{\tr\Phi(q_0)}\int_{B_\mu}\tr\Phi \geq  \frac{m}{(2\,c)^{m-1}}.
\end{equation*}
Theorem \ref{thm_majorar} is proved.

\subsection{Proof of Theorem \ref{bound_geom}}
Since $\bar M$ has bounded geometry, there exist constants $c>0$ and $\bar R_0>0$ such that the sectional curvature of $\bar M$ satisfies $K_{\bar M}\leq c^2$ and the injectivity radius satisfies $\bar R_q\geq \bar R_0$, for all $q\in \bar M$. We consider the constant functions $\m K(t)=c^2$ and $\al(t)=\ka\geq 0$, for all $t$. The function $h(t)=\frac{1}{c}\sin(c\,t)$, with $t\in (0,\frac{\pi}{c})$, is the maximal positive solution of (\ref{cauchy-K}). We take $0<t_0\leq \frac{\pi}{2\,c}$ the maximal positive number satisfying:
\begin{equation*}
h'(t)=\cos(c\,t)>\frac{\ka}{c} \sin(c\,t)=\al(t)\, h(t), 
\end{equation*}
for all $0<t<t_0$.
Since $0=\al'(t)\geq -\frac{h'(t)^2}{h(t)^2}-\m K(t)$, for all $t\in (0,\frac{\pi}{2c}]$, we obtain that $\mu_{\m K,\al}=t_0$. Let $E$ be an end of $M$ and $\la:E\to [0,\infty)$ a nonnegative $C^1$ function. The operator $\Phi(v)=\la(q)v$, for all $q\in E$ and $v\in T_q M$, satisfies $|\Phi (\na r)|=\la|\na r|\leq \la=\frac{\tr\Phi}{m}$, since $|\na r|\leq 1$ and $\tr\Phi=m\la$.  Thus Theorem \ref{phi r} applies. Thus, for all $0<\mu< \min\{\mu_{\m K,\al},\bar R_0\}$ and $q_0\in E$ such that $B_{\mu}(q_0)\subset E$, the following holds:
\begin{equation}\label{estimate-ball}
\int_{B_\mu(q_0)} \la \geq \la(q_0)\,\Ga(\mu),
\end{equation}
where $\Ga(\mu)=\frac{m}{c^{m-1}}\int_0^\mu \sin(c\,\tau)^{m-1}e^{-m\ka\tau}d\mu > 0.$

Assume that $\limsup_{x\to \infty \atop{x\in E}} \la(x)>0$. Then there exists $\de>0$ and a sequence $(q_1,q_2,\ldots)$ of points in $E$, with $d(q_k,x_0)\to \infty$, where $x_0$ is a fixed point of $M$, satisfying that $\la(q_k)\geq \de>0$, for all $k$. 
Fixed $0<\mu_0<\min\{\mu_{\m K,\al}, \bar R_0\}$, after a subsequence, we can assume that $B_{\mu_0}(q_k)\subset E$ and $B_{\mu_0}(q_k)\cap B_{\mu_0}(q_l)=\emptyset$, for all $k\neq l$. Thus, by (\ref{estimate-ball}), we have that
$\int_E \la \geq \sum_{k=1}^{N} \int_{B_{\mu_0}(q_k)}\la \geq N\de\,\Ga(\mu_0),$
for all integer $N\geq 1$. This implies that $\int_E \la=+\infty$. Theorem \ref{bound_geom} is proved.

\section*{Acknowledgement} The authors thank Walcy Santos and Detang Zhou for helpful suggestions during the preparation of this article.



\begin{thebibliography}{99}

\bibitem{asz} Alencar, H., Santos, W. and Zhou, D., {Curvature integral estimates for complete hypersurfaces}, to appear in Illinois Math. Journal.

\bibitem{bc} Barbosa, J. L. M. and Colares, A. G., {\it Stability of hypersurfaces with constant $r$-mean curvature}, Ann. Global Anal. Geom. 15 (1997), no. 3, 277 -- 297.

\bibitem{csz} Cao, H.-D., Shen, Y. and Zhu, S., {\it The structure of stable minimal hypersurfaces in $R^{n+1}$}. Math. Res. Lett. 4 (1997), no. 5, 637 -- 644.

\bibitem{cy} Cheng, S.-T. and Yau, S.-T., {\it Hypersurfaces with constant scalar curvature}, Math. Ann. 255 (1977), 195 -- 204. 

\bibitem{ccz} Cheng, X., Cheung, L.-F. and Zhou, D., {\it The structure of weakly stable constant mean curvature hypersurfaces}. Tohoku Math. J. (2) 60 (2008), no. 1, 101 -- 121.

\bibitem{cwx} do Carmo, M. P., Wang, Q., Xia, C., {\it Complete submanifolds with bounded mean curvature in a Hadamard manifold}. J. Geom. Phys. 60 (2010), no. 1, 142 -- 154.

\bibitem{fe} Federer, H., {\it Geometric Measure Theory}, Springer-Verlag New York Inc., New York, 1969,
Die Grundlehren der mathematischen Wissenschaften, Band 153.

\bibitem{fe2} Federer, H., {\it Curvature Measures}, Trans. Amer. Math. Soc. 93 (1959), no. 3, 418 -- 491.

\bibitem{fr} Frensel, K. R., {\it Stable complete surfaces with constant mean curvature}, Bol. Soc. Brasil. Mat. (N.S.) 27 (1996), no. 2, 129 -- 144.

\bibitem{gw} Greene, R. E. and Wu, H., {\it Function theory on manifolds which possess a pole}, Lecture Notes in Mathematics, vol. 699, Springer, Berlin, 1979.

\bibitem{re} Reilly, R., {\it Variational properties of functions of the mean curvatures for hypersurfaces in space forms}, J. Differential Geometry 8 (1973), 465 -- 477.

\bibitem{ro} Rosenberg, H., {\it Hypersurfaces of constant curvature in space forms}, Bull. Sci. Math. 117 (1993), no. 2, 211 -- 239.
\end{thebibliography}
\end{document}